\documentclass[12pt]{article}
\usepackage{graphicx}
\usepackage{subfigure}
\usepackage{booktabs}   
\usepackage{hhline}     
\usepackage{epstopdf} 

\usepackage{booktabs}
\usepackage{color}
\usepackage{latexsym}
\usepackage{setspace}
\usepackage{amsmath}
\usepackage{amssymb}
\usepackage{latexsym}
\usepackage{amsmath}
\usepackage{mathrsfs}
\usepackage{amssymb}
\usepackage[numbers]{natbib}
\usepackage{hyperref}
\hypersetup{
	colorlinks=true,
	linkcolor=blue,
	filecolor=blue,
	urlcolor=blue,
}
\usepackage{subfigure}
\usepackage{graphics}
\usepackage{graphicx}
\usepackage{multirow}

\makeatletter \oddsidemargin  -.5cm \evensidemargin -.5cm \textwidth
17.5cm \topmargin -.65in \textheight 24cm
\makeatother   

\newcommand{\bd}{\begin{definition}}
	\newcommand{\ed}{\end{definition}}
\newcommand{\br}{\begin{remark}}
	\newcommand{\er}{\end{remark}}
\newcommand{\bea}{\begin{eqnarray}}
	\newcommand{\eea}{\end{eqnarray}}
\newcommand{\beann}{\begin{eqnarray*}}
	\newcommand{\eeann}{\end{eqnarray*}}
\newtheorem{theorem}{Theorem}[section]

\newtheorem{lemma}[theorem]{Lemma}

\newtheorem{remark}{Remark}[section]
\newtheorem{example}{Example}[section]

\numberwithin{equation}{section} \numberwithin{equation}{section}

\begin{document}
	\title{Improved estimation of differential entropy of a multivariate Gaussian distribution}
	\author{%
		Somnath Mondal$^{*1}$ and Lakshmi Kanta Patra$^{1}$\\[8pt]
		{\small $^{1}$\textit{Department of Mathematics, Indian Institute of Technology Bhilai, India}}}%
	\date{}
	\maketitle
	\renewcommand{\thefootnote}{\fnsymbol{footnote}}
	\footnotetext[1]{\texttt{somnathmondal@iitbhilai.ac.in}}
	
	\begin{abstract}
	We consider estimation of entropy of a multivariate variate normal distribution with unknown mean vector and unknown covariance matrix. Using a decision-theoretic approach, we study the problem under a general location-invariant loss function that is strictly convex and minimized at zero. We first characterize the class of affine equivariant estimators and obtain the best affine equivariant estimator. We further show that the best affine equivariant estimator is inadmissible under the general loss. Exploiting information in the sample mean, we derive Stein-type improved estimators, including an extension to the case of a mean matrix. Under the linex loss, we obtain a smooth Brewster--Zidek-type estimator that dominates the best affine equivariant estimator, and we show that this estimator is generalized Bayes under a suitable non-informative prior.
	\end{abstract}
	
	
	
	\section{Introduction} 
In information theory, a well-known measure of uncertainty of a probability distribution is the entropy. Shannon used the term entropy in $1948$ within the communication theory and denotes the average quantity of information generated by stochastic source data. The method for evaluating thermodynamic properties like entropy is an important problem. It is especially applied in the fields of molecular biology, chemistry, and molecular physics \cite{jusuf2002role}.  Let $Y$ be a random variable with probability density function $\Lambda(y|\theta).$ Shannon's entropy is defined as
\begin{eqnarray*}\label{eq0.1}
	H(\theta)=-E(\ln \Lambda(Y|\theta))
\end{eqnarray*}
Shannon's entropy is used in various areas of science and technology, such as molecular bio-sciences, hydrology, and software reliability.  In the literature, several authors studied the estimation of entropy in non-parametric problems. However, very few authors have studied the estimation of entropy using parametric inference techniques, mainly using the decision-theoretic approach. Over the past two decades, some researchers have used a decision-theoretic approach to estimate the entropy of various continuous probability distributions. \cite{misra2005estimation} studied the estimation of entropy of a multivariate normal distribution with respect to squared error loss. They have calculated the bias and risk values of the maximum likelihood estimator (MLE)  and best affine equivariant estimator (BAEE). \cite{misra2005estimation} proved the inadmissibility of BAEE by proposing improved estimators. To prove this, they have obtained improved \cite{stein1964} and \cite{brewster1974improving} type estimators. The authors also demonstrated that the \cite{brewster1974improving} type improved estimators are generalized Bayes. \cite{kayal2013estimation} have considered the estimation of entropy of several exponential distributions with common scale parameters and different location parameters under the squared error loss function. They have proposed sufficient under which BAEE is inadmissible.
As a consequence, the authors have obtained an improved estimator. Estimation Renyi entropy of several exponential populations has been studied by \cite{kayal2015estimating} with respect to the squared error loss function. They have derived UMVUE. The authors proposed sufficient conditions for improvement over affine and scale equivariant estimators. As an application, they derived estimators that dominate UMVUE and MLE. \cite{patra2020estimating} studied the estimation entropy of an exponential distribution under a general location invariant loss function. The authors prove the inadmissibility of BAEE by proposing various improved estimators. Improved estimation entropy of an exponential population based on doubly censored sample considered by \cite{patra2020improved}. They have proposed various estimators that dominate BAEE. Improved estimation of the entropy of a scale mixture of exponential distribution has been discussed by \cite{petropoulos2020improved}. They have obtained estimators that dominate BAEE. Applying these results, the authors derived improved estimators for the entropy of a multivariate Lomax distribution. Recently \cite{misra2023shrinkage} discuss the problem of finding improved estimators selected entropy for two gamma distributions. For this problem, they have considered the squared error loss function. Several shrinkage estimators are obtained that dominate naive estimators. From the above literature review, we have seen that there is not much work for estimating the entropy of multivariate distribution except \cite{misra2005estimation}.

In this paper, we consider the problem of estimating entropy of a multivariate normal distribution. Let $X_1,X_2,\dots,X_n$ be a random sample form a $N_p(\eta,\boldsymbol{\Sigma})$ having density
\begin{eqnarray}
	f(x;\eta,\boldsymbol{\Sigma})=\frac{1}{(2\pi)^{p/2}|\boldsymbol{\Sigma}|^{1/2}}
	\exp\left(-\frac{1}{2}(x-\eta)^t\boldsymbol{\Sigma}^{-1}
	(x-\eta)\right).
\end{eqnarray}
The entropy of a $p-$variate normal random variable $X$ with a probability density function $f_{\eta,\boldsymbol{\Sigma}}(x)$, where $\eta\in\mathbb{R}^p$ is the mean vector and $\boldsymbol{\Sigma}$ is the $p\times p$ positive definite covariance matrix, is provided by
\begin{align*}
	H_p(\boldsymbol{\Sigma})&=E\left(-\ln f_{\eta,\boldsymbol{\Sigma}}(X)\right)\\
	&=\frac{p}{2}\left[1+\ln(2\pi)\right]+\frac{\ln|\boldsymbol{\Sigma}|}{2}.
\end{align*}
So estimating $H_p(\boldsymbol{\Sigma})$ is equivalent to estimating $\theta=\ln |\boldsymbol{\Sigma}|$ of multivariate normal distribution. A complete sufficient statistics is 
\begin{align*}
	\bar{X}=\frac{1}{m} \sum_{i=1}^{m}X_i, ~ X=\sqrt{m}\bar{X} ~\mbox{and}~ S=\sum_{i=1}^m(X_i-\bar{X})(X_i-\bar{X})^t.
\end{align*}
Then we have 
\begin{eqnarray}
	X\sim N_p(\mu,\boldsymbol{\Sigma})~~~S\sim W_p(m-1,\boldsymbol{\Sigma}),~~ \mbox{where}~~\mu=\sqrt{m}\eta,~n=m-1.
\end{eqnarray}
are independently distributed, here $W_p(n,\boldsymbol{\Sigma})$ denotes the p-variate Wishart distribution with $m-1$ degrees of freedom and covariance matrix $\boldsymbol{\Sigma}$. In this paper we will consider the estimation of $\theta=\ln |\boldsymbol{\boldsymbol{\Sigma}}|$ with respect to location invariant loss function $L(\delta-\theta)$. Assume that the loss function $L(t)$
under consideration satisfies the following
conditions
\begin{itemize}
	\item[(A1)] $L(t)$ is twice differentiable and
	strictly convex function in $\mathbb{R},$ where $\mathbb{R}$ denotes
	the set of real numbers. Further, $L(t)$ is minimized at $t=0$ with
	$L(0)=0$.
	\item[(A2)] Integrals involving $L(t)$ are finite and differentiation under the integral sign is permissible.
\end{itemize}
The remaining sections of this paper are arranged as follows. In Section \ref{se2}, we derive an estimator that improves upon the BAEE. Moreover we have obtain a generalized Bayes estimator which coincides with BAEE under an improper prior. Furthermore in Section \ref{se3}, we obtained a improved estimator using Stein type approach. Further we have extended the problem of estimation of entropy for matrix variate normal distribution in Section \ref{se4}. In Section \ref{se5} we have obtained a smooth improved estimator for linex loss function. Lastly, we give a conclusion in Section \ref{se6}.

\section{Best affine equivariant estimator}\label{se2}
In this section, we derive BAEE and its properties. Consider the affine group $Al_{n}$ that is the collection of pairs $(\boldsymbol{A}, b)$, where $\boldsymbol{A}$ is a $p \times p$ non singular matrix, $b \in \mathbb{R}^p$. The group $Al_p$ act on the  $\mathscr{X}=\mathbb{R}^p\times \mathcal{S}_p^+$  in the following manner
$$(\boldsymbol{A}, b)(X,S)=(\boldsymbol{A}X+b,\boldsymbol{A}S\boldsymbol{A}^t)~~~(\boldsymbol{A}, b)(\mu,\boldsymbol{\Sigma})=(\boldsymbol{A}\mu+b,\boldsymbol{A\Sigma A}^t),$$
where $\mathcal{S}_p^+$ set of all $p\times p$ positive definite matrix. The estimation problem is invariant under the affine group $Al_p$. Under this group action we have $\ln |\boldsymbol{\Sigma}| \rightarrow \ln |\boldsymbol{\Sigma}|+ \ln |\boldsymbol{A}|^2$ for all $p \times p$ non-singular real matrices $\boldsymbol{A}$ and all $b \in \mathbb{R}^p$. An estimator  $\delta(X,S)$ is called $Al_n$- equivariant or affine equivariant estimator if $\delta(X,S)$ satisfies 
$$\delta(\boldsymbol{A}X+b,\boldsymbol{A}S\boldsymbol{A}^t)=\delta(X,S)+\ln |\boldsymbol{A}|^2$$
\begin{lemma}\rm
	The form of an affine equivariant estimator of $\theta$ is given as $\delta_c(X,S)=\ln|S|-c$ for some real constant $c$.
\end{lemma} 
\textbf{Proof.} An affine equivariant estimator satisfies 
\begin{equation} \label{affine}
	\delta(\boldsymbol{A}X+b,\boldsymbol{A}S\boldsymbol{A}^t)=\delta(X,S)+\ln |\boldsymbol{A}|^2.
\end{equation}
We have $S$ is a positive definite matrix so there exist a positive definite matrix $S^{1/2}$ such that $S=S^{1/2}S^{1/2}$. We take $b=-\boldsymbol{A}X$ and $\boldsymbol{A}=\boldsymbol{H}S^{-1/2}$ in (\ref{affine}), where $\boldsymbol{H}$ is an $p\times p$ orthogonal matrix and $S^{-1/2}$ is the inverse of $S^{1/2}$. Then we get the form of any affine equivariant estimator as 
$$\delta_c(X,S)=\ln|S|+c$$
\begin{remark}
	If $c=p\ln n$, then we get the MLE of $\ln |\boldsymbol{\Sigma}|$ as $\delta_{ml}\left(X,S\right)=\ln|\frac{1}{n}S|$ is mle of $\ln \boldsymbol{\Sigma}$. So the MLE belongs to the class of affine equivariant estimator. The following theorem gives the best affine equivariant estimator of $\ln |\boldsymbol{\Sigma}|$.
\end{remark}
\begin{theorem}
	Under the location invariant loss function $L(\delta-\theta)$ the best affine equivariant estimator is given as $\delta_0(X,S)=\ln|S|+c_0$, where $c_0$ is the solution of 
	\begin{eqnarray}\label{eq:1}
		E \left[L^{\prime}(\ln V+c_0)\right]=0, 
	\end{eqnarray}
	with $V=\frac{|S|}{|\boldsymbol{\Sigma}|}\sim\prod \limits _{i=1}^{p} \chi^2_{m-i}$.
\end{theorem}
\textbf{Proof.} The risk of $\delta_c(X,S)$ is obtained as 
\begin{eqnarray*}
	R(c)&=&E L(\ln |S|+c-\ln |\boldsymbol{\Sigma}|) = E L\left( \ln \frac{|S|}{|\boldsymbol{\Sigma}|}+c \right) = E L\left( \ln \frac{|S|}{|\boldsymbol{\Sigma}|}+c \right).
\end{eqnarray*}
Since, the $L(t)$ is a strictly bowl-shaped loss function so $L(t)$ is convex. This proves the result.
\begin{example}
	Under the quadratic loss function $L_1(t)=t^2$, the BAEE of $\ln|\boldsymbol{\Sigma}|$ is given by $\delta_{01}\left(X,S\right)=\ln|S|+c_{01}$, where $c_{01}=- p \ln 2 - \sum\limits_{i=1}^p \psi( \frac{m-i}{2})$. Moreover, $\delta_{01}\left(X,S\right)$ is also an unbiased estimator.
\end{example}
\begin{example}
	For the linex loss function $L_2(t)=e^{at}-at-1$, $a\in \mathbb{R}/\{0\}$, the BAEE of $\ln |\boldsymbol{\Sigma}|$ is given by $\delta_{02}(X,S)=\ln|S|+c_{02}$, where $c_{02}=-p\ln2 -\frac{1}{a}\sum\limits_{i=1}^{p}\ln \left(\frac{\Gamma \left(\frac{m-i}{2}+a\right)}{\Gamma \left(\frac{m-i}{2}\right)}\right)$.
\end{example}
\begin{remark} The best affine equivariant estimator $\delta_{01}$ is also an unbiased estimator of $\ln|\boldsymbol{\Sigma}|$, which is rare statistical phenomenon. It follows that the BAEE $\delta_{01}$ is better alternative to $\delta_{02}$ in terms of bias.
\end{remark}
\begin{remark}
	The expression for the bias of the estimator $\delta_{02}$ is $B(\delta_{02})=\displaystyle \sum_{i=1}^{p}\psi\left(\frac{m-i}{2}\right)-\frac{1}{a}\sum_{i=1}^{p}\ln\left( \frac{\Gamma\left(\frac{m-i}{2}+a\right)}{\Gamma\left(\frac{m-i}{2}\right)}\right)$. 
\end{remark}
\begin{theorem}
	Let  $c_{01}= p\ln2+\sum_{i=1}^{p}\psi\left(\frac{m-i}{2}\right)$, $c_{02}= -\frac{1}{a}\sum_{i=1}^{p}\ln\left[2^a\frac{\Gamma \left(\frac{m-i}{2}+a\right)}{\Gamma \left(\frac{m-i}{2}\right)}\right]$, let $\psi^{(1	)}(z)=\frac{d}{dz}\left(\psi(z)\right)$ denote the trigamma function. Then
	\begin{enumerate}
		\item[(i)] $\text{B}\left( \delta_{02} \right) = \sum_{i=1}^{p}\psi(\frac{m-i}{2})-\frac{1}{a}\sum_{i=1}^{p}\ln\left[ \frac{\Gamma\left(\frac{m-i}{2}+a\right)}{\Gamma\left(\frac{m-i}{2}\right)}\right]$
		\item[(ii)] $R\left(\delta_{02}\right)=-a\sum_{i=1}^{p}\psi(\frac{m-i}{2})-\sum_{i=1}^{p}\ln\left[ \frac{\Gamma\left(\frac{m-i}{2}+a\right)}{\Gamma\left(\frac{m-i}{2}\right)}\right]$
	\end{enumerate}
\end{theorem}
\textbf{Proof.}
\begin{enumerate}
	\item [(i)]We know that $\text{B}\left( \delta_{02} \right)=E(\delta_{02}) - \ln|\boldsymbol{\Sigma}|$. To find the value of the bias of $\delta_{02}$ we have to calculate the average value of $\delta_{02}$. Now 
	\begin{align*}
		E(\delta_{02})=E\ln |S|-\frac{1}{a}\sum_{i=1}^{p}\ln\left(2^a\frac{\Gamma\left(\frac{m-i}{2}+a\right)}{\Gamma\left(\frac{m-i}{2}\right)}\right).
	\end{align*} 
	Hence we have
	\begin{align*}
		\text{B}\left( \delta_{02} \right)=&E\ln |S/\boldsymbol{\Sigma}|-\frac{1}{a}\sum_{i=1}^{p}\ln\left(2^a\frac{\Gamma\left(\frac{m-i}{2}+a\right)}{\Gamma\left(\frac{m-i}{2}\right)}\right)\\
		=&E\left(\prod_{i=1}^{p}\chi_{m-i}^2\right)-\frac{1}{a}\sum_{i=1}^{p}\ln\left(2^a\frac{\Gamma\left(\frac{m-i}{2}+a\right)}{\Gamma\left(\frac{m-i}{2}\right)}\right)\\
		=&\sum_{i=1}^{p}E\left(\ln \chi_{m-i}^2\right)-\frac{1}{a}\sum_{i=1}^{p}\ln\left(2^a\frac{\Gamma\left(\frac{m-i}{2}+a\right)}{\Gamma\left(\frac{m-i}{2}\right)}\right)\\
		=&\sum_{i=1}^{p}\psi\left(\frac{m-i}{2}\right)-\frac{1}{a}\sum_{i=1}^{p}\ln\left(\frac{\Gamma\left(\frac{m-i}{2}+a\right)}{\Gamma\left(\frac{m-i}{2}\right)}\right)
	\end{align*}
	This completes the proof.
	\item [(ii)] We analyze the risk function of the estimator $\delta_{02}(X,S)$ is
	\begin{equation}\label{eq2.3}
		\begin{aligned}
			R(\delta_{02},\theta)=& E\left[e^{a\left(\delta_{02}-\ln|\Sigma|\right)}-a\left(\delta_{02}-\ln|\Sigma|\right)-1\right]\\
			=&E\left[e^{a\left(\ln\frac{|S|}{|\Sigma|}+c_{02}\right)}-a\left(\ln\frac{|S|}{|\Sigma|}+c_{02}\right)-1\right]\\
		\end{aligned}
	\end{equation}
	To find the risk we have to use the following expected values
	
	$E\left[ \prod_{i=1}^{p}\chi_{m-i}^2\right]^a= \prod_{i=1}^{p}2^a \frac{\Gamma\left(\frac{m-i}{2}+a\right)}{\Gamma\left(\frac{m-i}{2}\right)}~\mbox{see \cite{rohatgi2015introduction}},~\sum_{i=1}^{p}E\ln \chi_{m-i}^2=p\ln2 +\sum_{i=1}^{p}\psi\left(\frac{m-i}{2}\right) ~\mbox{and}~ac_{02}=-\sum_{i=1}^{p}\ln\left[2^a\frac{\Gamma\left(\frac{m-i}{2}+a\right)}{\Gamma\left(\frac{m-i}{2}\right)}\right]$.
	Putting these values in the above equation \ref{eq2.3} we get
	\begin{align*}
		E\left(\delta_{c_{02}},\theta \right)=- a\sum_{i=1}^{p}\psi(\frac{m-i}{2}) - \sum_{i=1}^{p}\ln\left[ \frac{\Gamma\left(\frac{m-i}{2}+a\right)}{\Gamma\left(\frac{m-i}{2}\right)}\right]
	\end{align*}
	This completes the proof.
\end{enumerate}

\subsection{Generalized Bayes estimator of $\ln|\boldsymbol{\Sigma}|$}
In this subsection, we derive the generalized Bayes estimator of $\ln|\boldsymbol{\Sigma}|$ under the loss functions $L_1(t)$ and $L_2(t)$. We also show that the generalized Bayes estimator of $\ln|\boldsymbol{\Sigma}|$ coincides with the BAEE when the prior is chosen to be the same as that used by \cite{misra2005estimation}. Consider an improper prior on $(\mu, \boldsymbol{\Sigma})$ whose density is 
\begin{align*}
	\pi(\mu,\boldsymbol{\Sigma})=\frac{1}{|\boldsymbol{\Sigma}|^{\frac{p+1}{2}}}, ~~ \boldsymbol{\Sigma}>0, ~ \mu \in \mathbb{R}^p.
\end{align*}
\begin{theorem}
	Under the quadratic loss function $L_1(t)$, the generalized Bayes estimator $\delta_{G1}$ of $\ln|\boldsymbol{\Sigma}|$ with respect to the prior density $\pi(\mu,\boldsymbol{\Sigma})$ is the BAEE $\delta_{01}$.
\end{theorem}
\noindent\textbf{Proof.} The proof of this theorem can be obtained using the same argument as in \cite{misra2005estimation}.
\begin{theorem}
	Under the linex loss function $L_2(t)$, the generalized Bayes estimator $\ln|\boldsymbol{\Sigma}|$, with respect to the prior density $\pi(\mu,\boldsymbol{\Sigma})$, is the BAEE $\delta_{02}$.
\end{theorem}
\noindent\textbf{Proof.} As, for given $(\mu,\boldsymbol{\Sigma})$, the statistics $X$ and $S$ are independent distributed as $N_p(\mu,\boldsymbol{\Sigma})$ and $W_p(m-1,\boldsymbol{\Sigma})$, then the posterior density of $(\mu,\boldsymbol{\Sigma})$, given $(X,S)$ is 
\begin{align*}
	\frac{m^{\frac{p}{2}}|S|^\frac{m-1}{2} }{(2\pi)^{\frac{p}{2}}2^{\frac{(m-1)p}{2}}\Gamma_p\left(\frac{m-1}{2}\right)}\frac{1}{|\boldsymbol{\Sigma}|^{\frac{m+p+1}{2}}}e^{-\frac{1}{2}\left(X-\mu\right)^t\boldsymbol{\Sigma}^{-1}\left(X-\mu\right)-\frac{1}{2}\mbox{tr}\left(\boldsymbol{\Sigma}^{-1}S\right)},~~~~~~~ \mu\in \mathbb{R}^p, \ \boldsymbol{\Sigma>0}\\
\end{align*}
where $\Gamma_p(\cdot)$ denotes the p-variate gamma function. Hence, we have for given $(X,S)$, $\boldsymbol{\Sigma}^{-1} \sim W_p(m-1,S^{-1})$ (see Problem 3.6 in \cite{muirhead2009aspects}). Using Theorem 3.2.15 in \cite{muirhead2009aspects} we have 
\begin{align*}
	\frac{|\boldsymbol{\Sigma}^{-1}|}{|S^{-1}|} \sim \prod_{i=1}^p \chi_{m-i}^2
\end{align*}
For the linex loss function, the generalized Bayes estimator of $\ln|\Sigma|$ is given by 
\begin{align*}
	\delta_{G2}=\ln|S|-\frac{1}{a}\ln E|V|^a=\delta_{02}.
\end{align*}
Thus we have the result.
\section{Improvement over the BAEE}\label{se3}
In this section we will propose estimators that dominate BAEE. To find improved estimators we consider the class of estimator as $$\delta_{\psi}(X,S)=\ln|S+XX'|+\psi(X'S^{-1}X)$$
The risk function of $\delta_{\psi}(X,S)$ is 
\begin{eqnarray*}
	R(\delta_{\psi},\theta)&=&EL\left\{\ln|S+XX'|+\psi(X'S^{-1}X)-\ln|\boldsymbol{\Sigma}|\right\}=EL\left\{\ln\frac{|S+XX'|}{|\boldsymbol{\Sigma}|}+\psi(X'S^{-1}X)\right\}
\end{eqnarray*}	To simplify the expression and calculation we use the distribution theory from \cite{shorrock1976improved}. We write 
	
	$$	\begin{array}{ll}
		X= \boldsymbol{A}U,& S=\boldsymbol{A}W\boldsymbol{A}',\\
		\mu =\boldsymbol{A}\mu_0,& \boldsymbol{\Sigma}= \boldsymbol{A}\boldsymbol{I} \boldsymbol{A}'
	\end{array}$$
	where $\boldsymbol{A}$ is any $p\times p$ nonsingular real matrix, $U\sim N(\mu_0,I)$, $W\sim W_p(n,I)$ and $\mu_0'=(\lambda^{\frac{1}{2}},0,\dots,0)$. Then risk function of any equivariant estimator $\psi(X,S)$ is 
	\begin{equation}
		R_{\psi}(\lambda)=E^{\lambda}L\left\{\ln|W+UU'|+\psi(T)\right\}
	\end{equation}
	where $T=U'W^{-1}U$. Let $\boldsymbol{O}$ be a random orthogonal matrix such that $\tilde{U}=\boldsymbol{O}U=||\tilde{U}||(1,0,\dots,0)$, and set $\tilde{W}=\boldsymbol{O}W\boldsymbol{O}'$. Then $||\tilde{U}||$ and $\tilde{W}$ are independent with $||\tilde{U}||^2\sim\chi^2_p(\lambda)$ and $\tilde{W}\sim W_p(n,I)$ and $T=U'W^{-1}U=\tilde{U}'\tilde{W}^{-1}\tilde{U}$, i.e. $T=||\tilde{U}||^2/(\tilde{W_{11}}-\tilde{W_{12}}\tilde{W_{22}}^{-1}\tilde{W_{21}})$, and $|W+UU'|=|\tilde{W}+\tilde{U}\tilde{U}'|=(1+T)(\tilde{W_{11}}-\tilde{W_{12}}\tilde{W_{22}}^{-1}\tilde{W_{21}})$, where $\tilde{W}$ is partitioned as $\tilde{W}=\begin{bmatrix}
		\tilde{W_{11}}&\tilde{W_{12}}\\
		\tilde{W}_{21}&\tilde{W_{22}}
	\end{bmatrix}$ and $\tilde{W_{11}}$ is $1\times1$ matrix. Then by standard distribution theory we have, $$	\begin{array}{ll}
		u=||\tilde{U}||^2\sim\chi^2_{p+2k},\\
		v=\tilde{W_{11}}-\tilde{W_{12}}\tilde{W_{22}}^{-1}\tilde{W_{21}}\sim \chi^2_{n-p+1},\\
		w=|\tilde{W_{22}}|\sim  \prod_{i=1}^{p-1}\chi_{n-i+1}^2
	\end{array}$$ where $k$ follows a Poisson distribution with $Ek=\lambda/2$ and the above $\chi^2$ random variables are all conditionally independent given $k$. So the risk function becomes
	\begin{equation}
		R_{\psi}(\lambda)=E^{\lambda}E^{(\frac{u}{v},k)}L\left\{\ln(u+v)+\ln w+\psi\left(\frac{u}{v}\right)\bigg| \frac{u}{v},k\right\}
	\end{equation}
	Given that $u$, $v$, $w$ are independent then $u+v$, $w$ are independent also $\frac{u}{v}$, $w$ are independent. The above risk is minimize at $\psi^{*}$, where $\psi^{*}$ is the unique solution of 
	\begin{equation*}
		E^{(\frac{u}{v},k)}L'\left\{\ln(u+v)w+\psi^{*}\left(\frac{u}{v}\right)\bigg| \frac{u}{v},k\right\}=0
	\end{equation*}
	Let $f_k(h,w)$ is the joint density of $h=u+v$ and $w$. Now we can easily check that $\frac{f_k(h,w)}{f_0(h,w)}$ is increasing in $h$. Then we get $\psi_k(\frac{u}{v})\le\psi_0(\frac{u}{v})$, where $\psi_0$ is the solution of
	\begin{equation*}
		E_0L'\left\{\ln hw+\psi_0\left(\frac{u}{v}\right)\right\}=0
	\end{equation*}
	The above equation is equivalent to 
	\begin{equation*}
		\int_{0}^{\infty}\int_{0}^{\infty}\left\{\ln (hw)+\psi_0\left(\frac{u}{v}\right)\right\}dhdw=0
	\end{equation*}
	Now consider $\psi'(\frac{u}{v})=\min\left\{\psi_0(\frac{u}{v}),c_0\right\}$. Then $\psi_k(\frac{u}{v})\le\psi_0(\frac{u}{v})=\psi'(\frac{u}{v})\le c_0$ provided $P\left(\psi^{*}(\frac{u}{v})<c_0\right)>0$. Since the risk function is strictly bowl-shaped then we have risk function is strictly increasing on $\left(\psi_k,\infty\right)$. Thus we conclude that $R(\delta_{\psi'},\theta)<R(\delta_{c_0},\theta)$. Thus the improved estimator whose risk is lesser than usual one can be expressed as 
	$$\delta_{\psi_{0}}(X,S)=\ln|S+XX'|+\psi_{0}(X'S^{-1}X)$$
	\begin{theorem}
		The risk function of the estimator 	$$\delta_{\psi_{0}}(X,S)=\ln|S+XX'|+\psi_{0}(X'S^{-1}X)$$ is nowhere larger than the estimator $\delta_{\psi}(X,S)$ provided $P\left(\psi^{*}(\frac{u}{v})<c_0\right)>0$ under a general location invariant loss function $L(t)$, where $\psi_0$ be a solution of the equation
		\begin{equation*}
			E_0L'\left\{\ln hw+\psi_0\left(\frac{u}{v}\right)\right\}=0.
		\end{equation*}
	\end{theorem}
	\begin{example}
		\begin{enumerate}
			\item[(i)] For the quadratic loss function $L_1(t)$, the improved estimator of $\ln|\boldsymbol{\Sigma}|$ is given by
			\begin{align*}
				\delta_{\psi^1_0}(X,S)=&\ln|S+XX'|-\int_{0}^{\infty}\ln z f_Z(z)dz-\int_{\mathbb{R}^{p-1}}\ln w f_W(w)dw\\
				=&\ln|S+XX'|-\ln2+\psi\left(\frac{m}{2}\right)-p\ln2+\sum_{i=1}^{p-1}\psi\left(\frac{m-i}{2}\right)\\
				=&\ln|S+XX'|-(1+p)\ln2+\psi\left(\frac{m}{2}\right)+\sum_{i=1}^{p-1}\psi\left(\frac{m-i}{2}\right)
			\end{align*}
			\item[(ii)] Under the linex loss function $L_2(t)$, the improved estimator of $\ln|\boldsymbol{\Sigma}|$ is given as
			\begin{align*}
				\delta_{\psi^2_0}(X,S)=&\ln|S+XX'|-\frac{1}{a}\ln\int_{0}^{\infty}\int_{\mathbb{R}^{p-1}} (zw)^a f_{Z,W}(z,w)dzdw\\
				=&\ln|S+XX'| +\ln2 + \sum_{i=0}^{p-1}\ln\frac{\Gamma\left(\frac{m-i}{2}+a\right)}{\Gamma\left(\frac{m-i}{2}\right)}
			\end{align*}
		\end{enumerate}
	\end{example}
	
	
	\section{Improved Stein-type Estimator} \label{se4}
	In this section we generalize our above result on the the high dimension. We consider Independent random variables $\boldsymbol{X}(p \times m)\sim N(\boldsymbol{\mu},\boldsymbol{\Sigma}\bigotimes \boldsymbol{I})$ and $\boldsymbol{S}(p\times p)\sim W_p(n,\boldsymbol{\Sigma})$, $n\geq p$, where the mean matrix $\boldsymbol{\mu}$ is unknown and $\boldsymbol{\Sigma}$ is positive definite.
	
	In this section, we shall derive an estimator that is uniformly better than the BAEE $\delta_{c_0}$ considered in the last section.
	In this section, we will derive an estimator that dominates the BAEE $\delta_{c_0}$. In order to do so, we look at a larger family than the affine equivariant estimators. We consider one such class of estimators as  $\delta_{\zeta}\left(\boldsymbol{X},\boldsymbol{S}\right)=\ln|\boldsymbol{S}|+\zeta(\boldsymbol{W})$, where $\boldsymbol{W}=\frac{|\boldsymbol{S}+\boldsymbol{X}\boldsymbol{X}^{'}|}{|\boldsymbol{S}|}$ and $\zeta$ is a positive function.
	
	Then from \cite{shorrock1976improved} and \cite{sarkar1991stein}, conditionally on $\boldsymbol{K}=\kappa=\left(\kappa_1,\dots,\kappa_s\right)$, $|\boldsymbol{S}+\boldsymbol{X}\boldsymbol{X}^{'}|$ and $\boldsymbol{W}$ are independent and \begin{equation}
		|\boldsymbol{S}+\boldsymbol{X}\boldsymbol{X}^{'}|~|\boldsymbol{K}=\kappa \sim \prod_{i=1}^{p}\chi^2_{n-i+2+2\kappa_i}
	\end{equation}
	where $\kappa_{s+1}=\dots=\kappa_p=0$. We denote $f(.|r,\kappa)$ the density of the conditional distribution of $|\boldsymbol{S}|$ given that $\boldsymbol{W}=r$ and $\boldsymbol{K}=\kappa$.
	
	The following lemma taken from \cite{iliopoulos1999improving} which is useful for our derivation
	\begin{lemma}
		Conditionally on $\boldsymbol{W}=r$ and $\boldsymbol{K}=\kappa=\left(\kappa_1,\dots,\kappa_s\right)$,
		\begin{equation}
			|\boldsymbol{S}|=\frac{1}{r}\prod_{i=1}^{p} \chi^2_{n-i+2+2\kappa_i}
		\end{equation}
	\end{lemma}
	\begin{theorem}
		Let $b_0$ be the unique solution to the equation 
		\begin{equation}\label{eq4.3}
			E_{\boldsymbol{\mu}=\boldsymbol{0},\boldsymbol{\Sigma}=\boldsymbol{I}}\left\{ L'\left(\ln r+b_0\right) \right\}=0
		\end{equation}
		
		Then the estimator 
		\begin{equation}\label{eq4.4}
			\delta_{\boldsymbol{S}}= \ln|\boldsymbol{S}|-\min\{c_0, b_0+ \ln w\}.
		\end{equation}
	\end{theorem}
	\textbf{Proof.} 
	The risk function of $\delta_{\zeta}(\boldsymbol{X},\boldsymbol{S})$ is\begin{eqnarray*}
		R(\delta_{\zeta},\theta)&=&E\left[L\left(\ln|\boldsymbol{S}|+\zeta(\boldsymbol{W})-\ln|\boldsymbol{\Sigma}|\right)\right]=E\left[L\left(\ln\frac{|\boldsymbol{S}|}{|\boldsymbol{\Sigma}|}+\zeta(\boldsymbol{W})\right)\right]
		=E\left[L\left(\ln|\boldsymbol{S}^{*}|+\zeta(\boldsymbol{W})\right)\right]\\
		&=&E^{w,\kappa}E\left[L\left(\ln|\boldsymbol{S}^{*}|+\zeta(\boldsymbol{W},\boldsymbol{K})\right)\vert \boldsymbol{W}=w,\boldsymbol{K}=\kappa\right].
	\end{eqnarray*}
	The method of proof is based on \cite{stein1964} and \cite{brewster1974improving}. The function $c\rightarrow E\left\{L(x+c)|\boldsymbol{W}=w,\boldsymbol{K}=\kappa\right\}$ is strictly bowl-shaped and hence it minimized at a unique point $c=\phi(w,k)$ which satisfying $E\left\{L'(x+c)|\boldsymbol{W}=w,\boldsymbol{K}=\kappa\right\}=0$. We denote $R^{*}(c,\theta)=E^{w,\kappa}E\left[L\left(\ln|\boldsymbol{S}^{*}|+c\right)\vert \boldsymbol{W}=w,\boldsymbol{K}=\kappa\right].$
	Now we have to find the density of $|\boldsymbol{S}^{*}|$ given $\boldsymbol{W}=w$ and $\boldsymbol{K}=\kappa$. From \cite{iliopoulos1999improving} we get, $$|\boldsymbol{S}^{*}|\big\vert\boldsymbol{W}=w,\boldsymbol{K}=\kappa\sim\frac{1}{w}\prod_{i=1}^{p}\chi^2_{n-i+2\kappa_i+2}.$$ Now we have to calculate the density of $\frac{1}{w}\prod_{i=1}^{p}\chi^2_{n-i+2\kappa_i+2}.$ 
	Let $V= \prod_{i=1}^{p}\chi^2_{n-i+2\kappa_i+2}$, density of $V$ is $$f_V(v)=\frac{1}{2^{\sum_{i=1}^{ p}\frac{(n-i+2\kappa_i+2)}{ 2}} \prod_{i=1}^{p}\frac{n-i+2\kappa_i+2}{2}}e^{- p\frac{ v}{2}} v^{ \sum_{i=1}^{p}\frac{n-i+2\kappa_i+2}{2}-p}.$$
	Let $U=\frac{V}{w}$, then the density of the transformation is
	\begin{eqnarray*}
		f_U(u)&=&\frac{1}{2^{\sum_{i=1}^{p}\frac{(n-i+2\kappa_i+2)}{2}} \prod_{i=1}^{p}\frac{n-i+2\kappa_i+2}{2}}e^{-\frac{ puw}{ 2}}(uw)^{ \sum_{i=1}^{p}\frac{n-i+2\kappa_i+2}{2}-p}w\\
		&=&\frac{1}{2^{ \sum_{i=1}^{p}\frac{(n-i+2\kappa_i+2)}{2}} \prod_{i=1}^{p}\frac{n-i+2\kappa_i+2}{2}}e^{-\frac{puw}{2}}(u)^{\sum_{i=1}^{p}\frac{n-i+2\kappa_i+2}{2}-p}(w)^{\sum_{i=1}^{p}\frac{n-i+2\kappa_i+2}{2}-p+1}.
	\end{eqnarray*}
	Hence $R^{*}(c,k)=\int_{0}^{\infty}L(\ln u +c(w,k))f_{\boldsymbol{U}}(u)du.$\\
	Again take a transformation by $\boldsymbol{Y}=\ln \boldsymbol{U}$, hence we finally get the density of $\ln \boldsymbol{S}^{*}$ or $\ln \boldsymbol{U}$
	\begin{align*}
		f_{\boldsymbol{Y}}(y)=\frac{1}{2^{\sum_{i=1}^{p}\frac{(n-i+2\kappa_i+2)}{2}}\prod_{i=1}^{p}\frac{n-i+2\kappa_i+2}{2}}e^{-\frac{pe^yw}{2}}(e^y)^{\sum_{i=1}^{p}\frac{n-i+2\kappa_i+2}{2}-p+1}(w)^{\sum_{i=1}^{p}\frac{n-i+2\kappa_i+2}{2}-p+1}.
	\end{align*}
	Hence $R^{*}(c,k)=\int_{-\infty}^{\infty}L(\ln y +c(r,\kappa))f_{\boldsymbol{Y}}(y,w,\kappa)dy$, where $c(w,\kappa)$ is the unique solution of $R'(c,\kappa)=0$ i.e. $R'(c,k)=\int_{-\infty}^{\infty}L(\ln y +c)f_{\boldsymbol{Y}}(y,w,\kappa)dy=0$.\\
	Now we prove $c(w,\kappa)\leq c(w,0)$.
	\begin{eqnarray*}
		\frac{f(y,w,\kappa)}{f(y,w,0)}=\frac{e^{-\frac{pe^yw}{2}}e^{y\sum_{i=1}^{p}\frac{n-i+2\kappa_i+2}{2}-p+1}(w)^{\sum_{i=1}^{p}\frac{n-i+2\kappa_i+2}{2}-p+1}}{e^{-\frac{pe^yw}{2}}e^{y\sum_{i=1}^{p}\frac{n-i+2}{2}-p+1}(w)^{\sum_{i=1}^{p}\frac{n-i+2}{2}-p+1}}=e^{y\kappa}~~~ \text{is increasing in y}.
	\end{eqnarray*}
	By using Maximum likelihood ratio property we have  $c(w,\kappa)\leq c(w,0)$. where $c(w,0)$ satisfies the equation $\int_{-\infty}^{\infty} L'(y+c(w,0))f_Y(y,w,0)dy=0$.
	$$\int_{-\infty}^{\infty} L'(y+c(w,0))\frac{1}{2^{\sum_{i=1}^{p}\frac{n+2-i}{2}}\prod_{i=1}^{p}\Gamma(\frac{n+2-i}{2})}e^{-\frac{pwe^y}{2}}e^{y\sum_{i=1}^{p}(\frac{n+2-i}{2})-p+1}w^{\sum_{i=1}^{p}\frac{n+2-i}{2}-p+1}dy=0.$$
	Take transformation $e^y=z$ we get, 
	\begin{equation}
		\int_{0}^{\infty} L'(\ln z+c(w,0))\frac{1}{2^{\sum_{i=1}^{p}\frac{n+2-i}{2}}\prod_{i=1}^{p}\Gamma(\frac{n+2-i}{2})}e^{-\frac{pwz}{2}}z^{\sum_{i=1}^{p}(\frac{n+2-i}{2})-p}w^{\sum_{i=1}^{p}\frac{n+2-i}{2}-p+1}dz=0.
	\end{equation}
	Now recall the density of $Z=\prod_{i=1}^{p}\chi^2_{n+2-i}$ is $\frac{1}{2^{\sum_{i=1}^{p}\frac{n+2-i}{2}}\prod_{i=1}^{p}\Gamma(\frac{n+2-i}{2})}e^{-\frac{pz}{2}}z^{\sum_{i=1}^{p}(\frac{n+2-i}{2})-p}$.\\
	Now $Y=\frac{1}{w}Z=\frac{1}{w}\prod_{i=1}^{p}\chi^2_{n+2-i}$. The density of $Y$ is $h_Y(y)=\frac{1}{2^{\sum_{i=1}^{p}\frac{n+2-i}{2}}\prod_{i=1}^{p}\Gamma(\frac{n+2-i}{2})}e^{-\frac{pwy}{2}}y^{\sum_{i=1}^{p}(\frac{n+2-i}{2})-p}\\w^{\sum_{i=1}^{p}\frac{n+2-i}{2}-p+1}$.\\
	So finally we can write the equation $3.3$ as \begin{equation}
		\int_{0}^{\infty} L'(\ln z+c(w,0))f_Z(z,w)dz=0
	\end{equation} where $h_Z(z)\sim \frac{1}{w}\prod_{i=1}^{p}\chi^2_{n+2-i}$, and hence the density of $h(zw)$ is $$\frac{1}{2^{\sum_{i=1}^{p}\frac{n+2-i}{2}}\prod_{i=1}^{p}\Gamma(\frac{n+2-i}{2})}e^{-\frac{pwz}{2}} (zw)^{\sum_{i=1}^{p}(\frac{n+2-i}{2})-p}wdz.$$\\
	Put $zw=r$ Then $h_R(r)=\frac{1}{2^{\sum_{i=1}^{p}\frac{n+2-i}{2}}\prod_{i=1}^{p}\Gamma(\frac{n+2-i}{2})}e^{-\frac{pr}{2}} r^{\sum_{i=1}^{p}(\frac{n+2-i}{2})-p}dr$.
	Therefore 
	\begin{equation}
		\int_{0}^{\infty} L'(\ln r+\zeta_{0}(w)-\ln w)h_R(r)dr=0
	\end{equation}
	Take $b_0=\zeta_0(w)-\ln w$ then $\int_{0}^{\infty} L'(\ln r+b_0)h_R(r)dr=0$ where $h(r) \sim \prod_{i=1}^{p}\chi^2_{n+2-i}$.  By using Lemma 4.4 from \cite{iliopoulos1999improving} we get $b_0 <c_0$. Now we define $\zeta_{\boldsymbol{S}}(w)=\min \left\{c_0,\zeta_0(w)\right\}=\min \left\{c_0,b_0+\ln w\right\}$. 
	We define a function as $\delta_{\boldsymbol{S}}=\ln |\boldsymbol{S}| - \min \{c_0, \phi_0(w)\}=\ln|\boldsymbol{S}|-\min\{c_0, b_0+ \ln w\}$ we can conclude that $\delta_{\boldsymbol{S}}$ in $(\ref{eq4.4})$ has strictly smaller risk than $\delta_0$ because $R^{*}(c,\theta)$ is strictly bowl shaped, for any $c>c_0$, $R^{*}(c,\theta)$ is an increasing function of $c$. Consequently,$R^{*}(\phi_{0}(w),\theta)\leq R^{*}(c_0,\theta)$.
	\begin{remark}
		For the square error loss function $L_1(t)=t^2$, it follows from $(\ref{eq4.3})$ we have $b_0=-p\ln 2-\sum_{i=1}^{p}\psi(\frac{m+1-i}{2})$. Therefore, the improve estimator of $\ln|\Sigma|$ is obtain as
		\begin{equation*}
			\delta^{*}_{01}=\ln|\boldsymbol{S}|-\min \left\{p\ln2+\sum_{i=1}^{p}\frac{m-i}{2}, -p\ln 2-\sum_{i=1}^{p}\psi(\frac{m+1-i}{2})+\ln w\right\}.
		\end{equation*}
	\end{remark}
	\begin{remark}
		Under the linex loss function $L_2(t)=e^{at}-at-1, a\neq 0$, from  $(\ref{eq4.3})$ we get $b_0=\frac{1}{a}\ln\left[1-\prod_{i=1}^{p}\frac{\Gamma(\frac{1}{2}(n+2-i)+a)}{\Gamma(\frac{n+2-i}{2})}\right]$. Hence the estimator is \begin{equation*}
			\delta^{*}_{02}=\ln|\boldsymbol{S}|-\min\left\{\frac{1}{a}\sum_{i=1}^{p}\ln 2^a \frac{\Gamma \left(\frac{n-i}{2}+a\right)}{\Gamma \left(\frac{n-i}{2}\right)},-\frac{1}{a}\ln \prod_{i=1}^{p}\frac{\Gamma(\frac{1}{2}(n+2-i)+a)}{\Gamma(\frac{n+2-i}{2})}+\ln w\right\}.
		\end{equation*}
	\end{remark}
	
	\section{Improved brewster-zidek-type estimator under linex loss function} \label{se5}
	
	In this section following the idea of \cite{brewster1974improving}, we derive a smooth estimator of $\ln|\Sigma|$ that dominates the BAEE $\delta_0$. We therefore begin by considering estimators of the form
	\begin{equation}\label{bz1}
		\delta_{d,r}^{\star}(X,S)=
		\begin{cases}
			\ln|S|+c_0, & \text{for } 0\le T\le r,\\[2pt]
			\ln|S|+d, & \text{for } r<T\le 1,
		\end{cases}
	\end{equation}
	where $d\in(-\infty,\infty)$ and $r\in(0,1)$ are real constants. Then, for given $r\in(0,1)$ and $\lambda>0$, the risk function $R_\lambda(\delta^{\star}_{d,r})$ is minimized at $d\equiv d_r(\lambda)$, where
	\begin{equation}\label{bz2}
		d_r(\lambda)=-\frac{1}{a}\ln\left[\frac{E_\lambda\!\left(\left(\dfrac{|S|}{|\Sigma|}\right)^{a} I(r<T\le 1)\right)}{E_\lambda\big(I(r<T\le 1)\big)}\right].
	\end{equation}
	We now characterize the function $d_r(\lambda)$ through the following lemma, where $\lambda>0$ and $r\in(0,1)$.
	\begin{lemma}
		\begin{enumerate}
			\item[(i)] Let $\delta_{d,r}^{\star}(\cdot,\cdot)$ be as defined in $(\ref{bz1})$. Then, for fixed $r\in(0,1)$ and $\lambda>0$, $R_\lambda(\delta_{d,r}^{\star})$ is minimized at $d=d_r(\lambda)$, where $d_r(\lambda)$ is given by $(\ref{bz2})$.
			\item[(ii)] For fixed $r\in(0,1)$, $\sup\limits_{\lambda>0}d_r(\lambda)=d(r)$, where
			\begin{equation}\label{eq6.3}
				d(r)=-p\ln 2-\frac1a\sum_{i=1}^{p}\ln\frac{\Gamma\!\left(\frac{n-i+1}{2}+a\right)}{\Gamma\!\left(\frac{n-i+1}{2}\right)}
				-\frac1a\ln\left[\frac{\displaystyle\int_r^1 t^{\frac{n-p}{2}+a-1}(1-t)^{\frac p2-1}dt}{\displaystyle\int_r^1 t^{\frac{n-p}{2}-1}(1-t)^{\frac p2-1}dt}\right].
			\end{equation}
			\item[(iii)] $d(r)$ is an decreasing function of $r\in(0,1)$ and $c_{\star}<d(r)<c_{02}\ \forall\, r\in(0,1).$
		\end{enumerate}
	\end{lemma}
	\noindent \textbf{Proof.} $(i)$ Write $Z=\ln(|S|/|\Sigma|)$. The only part of $R_\lambda(\delta_{d,r})$ that depends on $d$ is	
	$$\begin{aligned}
		h_{r,\lambda}(d) & = E_\lambda\Big(\big\{e^{a(Z+d)}-a(Z+d)-1\big\}I(r<T\le 1)\Big)\\
		&=e^{ad}E_\lambda\big(e^{aZ}I(r<T\le1)\big)-aE_\lambda\big((Z+d)I(r<T\le1)\big)-E_\lambda\big(I(r<T\le1)\big).
	\end{aligned}$$
	We have $ h'_{r,\lambda}(d)=a\,e^{ad}E_\lambda\big(e^{aZ}I(r<T\le 1)\big)-a\,E_\lambda\big(I(r<T\le 1)\big),$ $
	h''_{r,\lambda}(d)=a^2e^{ad}E_\lambda\big(e^{aZ}I(r<T\le 1)\big)>0.$ So $h_{r,\lambda}$ is strictly convex. Setting $h'_{r,\lambda}(d)=0$ and dividing by $a\ne0$ gives
	\[
	e^{ad}=\frac{E_\lambda\big(I(r<T\le 1)\big)}{E_\lambda\big(e^{aZ}I(r<T\le 1)\big)}
	\quad \Longleftrightarrow \quad d=-\frac1a\ln\frac{E_\lambda\big(e^{aZ}I(r<T\le 1)\big)}{E_\lambda\big(I(r<T\le 1)\big)},
	\]
	which is $d_r(\lambda)$ given by (\ref{bz2}), since $e^{aZ}=(|S|/|\Sigma|)^a$.
	
	$(ii)$ Consider a fixed $r\in(0,1)$. Under the notation of Section \ref{sec3}, we have $|S|/|\Sigma|=|W|=T\prod\limits_{i=1}^{p}V_{ii}^2$. Recall that $(V_{11}^2,T),V_{22}^2,\ldots,V_{pp}^2$ are mutually independent where, $V_{ii}^2\sim\chi^2_{n-i+1}$ for $i=2,\ldots,p$, and for given $K=k$, $V_{11}^2\sim\chi^2_{n+2k}$ and $T\sim \mathrm{Beta}\big(\frac{n-p}{2},\frac p2+k\big)$ are independently distributed. Therefore from (\ref{bz2}) we have,
	\begin{align*}
		e^{-a\,d_r(\lambda)}
		&=\frac{E_\lambda\!\left(\left(\dfrac{|S|}{|\Sigma|}\right)^{a} I(r<T\le 1)\right)}{E_\lambda\big(I(r<T\le 1)\big)}
		=\frac{E_\lambda\!\left(\Big(T\prod_{i=1}^pV_{ii}^2\Big)^{a} I(r<T\le 1)\right)}{E_\lambda\big(I(r<T\le 1)\big)}\\[4pt]
		&=\frac{E_\lambda\big(L_1(K)\big)}{E_\lambda\big(L_2(K)\big)},\qquad \lambda>0,
	\end{align*}
	where $K$ is a Poisson random variable with mean $\lambda/2>0$ and, for $k=0,1,2,\ldots$,
	\begin{align*}
		L_1(k)&=2^{pa}\prod_{i=2}^{p}\frac{\Gamma\!\left(\frac{n-i+1}{2}+a\right)}{\Gamma\!\left(\frac{n-i+1}{2}\right)}\cdot
		\frac{\Gamma\!\left(\frac n2+k+a\right)}{\Gamma\!\left(\frac n2+k\right)}\\
		&\qquad\times\left[\beta\!\left(\frac{n-p}{2},\frac p2+k\right)\right]^{-1}\int_r^1 t^{\frac{n-p}{2}+a-1}(1-t)^{\frac p2+k-1} dt,\\[4pt]
		L_2(k)&=\left[\beta\!\left(\frac{n-p}{2},\frac p2+k\right)\right]^{-1}\int_r^1 t^{\frac{n-p}{2}-1}(1-t)^{\frac p2+k-1} dt ;
	\end{align*}
	here $\beta(\cdot,\cdot)$ denotes the beta function. We can write the above expression simply as
	\begin{equation}\label{eq:EK1}
		e^{-a\,d_r(\lambda)}=E_\lambda\big[L_3(U_1)\big],\qquad \lambda>0,
	\end{equation}
	where, for $k=0,1,\ldots$, $L_3(k)=L_1(k)/L_2(k)$ and $U_1$ is a random variable whose probability mass function is proportional to $\frac{(\lambda/2)^k L_2(k)}{k!}$, $k=0,1,\ldots$. We have
	\begin{equation}\label{eq:L3}
		L_3(k)=2^{pa}\prod_{i=2}^{p}\frac{\Gamma\!\left(\frac{n-i+1}{2}+a\right)}{\Gamma\!\left(\frac{n-i+1}{2}\right)}\cdot
		\frac{\Gamma\!\left(\frac n2+k+a\right)}{\Gamma\!\left(\frac n2+k\right)}\cdot L_4(k),\qquad k=0,1,\ldots,
	\end{equation}
	where $ L_4(k)=\frac{\displaystyle\int_r^1 t^{\frac{n-p}{2}+a-1}(1-t)^{\frac p2+k-1} dt}{\displaystyle\int_r^1 t^{\frac{n-p}{2}-1}(1-t)^{\frac p2+k-1} dt},\ k=0,1,\ldots.$
	For $s>0$ and $k=1,2,\ldots$, integrating by parts with $(1-t)^{\frac p2+k-1}$ as the differentiating factor and $t^{s-1}$ as the integrating factor, and then using $t^{s}=t^{s-1}-t^{s-1}(1-t)$, we get 
	$$ \int_r^1 t^{s-1}(1-t)^{\frac p2+k-1} dt=\frac{\left(\frac p2+k-1\right)\displaystyle\int_r^1 t^{s-1}(1-t)^{\frac p2+k-2} dt-r^{s}(1-r)^{\frac p2+k-1}}{s+\frac p2+k-1}.$$
	
	Applying this with $s=\frac{n-p}{2}+a$ in the numerator of $L_4(k)$ and with $s=\frac{n-p}{2}$ in its denominator, and after some standard adjustments, we get, for $k=1,2,\ldots$,
	\begin{multline*}
		L_4(k)=L_4(k-1)\,\frac{n+2k-2}{n+2k-2+2a}\\
		\times\left[\frac{p+2k-2-2r^{\frac{n-p}{2}+a}(1-r)^{\frac p2+k-1}\Big\{\displaystyle\int_r^1 t^{\frac{n-p}{2}+a-1}(1-t)^{\frac p2+k-2} dt\Big\}^{-1}}
		{p+2k-2-2r^{\frac{n-p}{2}}(1-r)^{\frac p2+k-1}\Big\{\displaystyle\int_r^1 t^{\frac{n-p}{2}-1}(1-t)^{\frac p2+k-2}dt\Big\}^{-1}}\right].
	\end{multline*}
	Note that $n+2k-2+2a>p+2k-2>0$, since $a>-\frac{n-p}{2}$. The numerator and the denominator of the term in square brackets are positive, because they equal
	$$ \frac{(n+2k-2+2a)\int_r^1 t^{\frac{n-p}{2}+a-1}(1-t)^{\frac p2+k-1}dt}{\int_r^1 t^{\frac{n-p}{2}+a-1}(1-t)^{\frac p2+k-2}dt}
	\quad\text{and}\quad
	\frac{(n+2k-2)\int_r^1 t^{\frac{n-p}{2}-1}(1-t)^{\frac p2+k-1}dt}{\int_r^1 t^{\frac{n-p}{2}-1}(1-t)^{\frac p2+k-2}dt},$$
	respectively. Since $t^{a}>r^{a}$ (respectively, $t^a<r^a$) for $t\in(r,1]$ when $a>0$ (respectively, $a<0$), we have, for all $r\in(0,1)$,
	\begin{equation}\label{eq:44L}
		\int_r^1 t^{\frac{n-p}{2}+a-1}(1-t)^{\frac p2+k-2}dt\ > (<)\ \ r^{a}\int_r^1 t^{\frac{n-p}{2}-1}(1-t)^{\frac p2+k-2}dt\qquad\text{according as } a>(<) 0 .
	\end{equation}
	Consequently, the term in square brackets is $>1$ if $a>0$ and $<1$ if $a<0$, so that
	$$ L_4(k)\ > (<)\ \frac{n+2k-2}{n+2k-2+2a}\,L_4(k-1)\qquad\text{according as } a > (<) 0.$$
	Now, using the fact that $\Gamma(z+1)=z\,\Gamma(z)$ for $z>0$, we get
	$$ \frac{\Gamma\!\left(\frac n2+k+a\right)}{\Gamma\!\left(\frac n2+k\right)}
	=\frac{n+2k-2+2a}{n+2k-2}\cdot\frac{\Gamma\!\left(\frac n2+k-1+a\right)}{\Gamma\!\left(\frac n2+k-1\right)},\qquad k=1,2,\ldots,$$
	and therefore, from (\ref{eq:L3}), for $k=1,2,\ldots$,
	$$L_3(k)\ > (<)\ L_3(k-1)\qquad\text{according as } a >(<) 0.$$		
	
	Thus, $L_3(k)$ is an increasing (decreasing) function of $k\in\{0,1,\ldots\}$ if $a>0$ ($a<0$). Now, let $U_2$ be a random variable that is degenerate at $0$. Then, for each $\lambda>0$, the random variable $U_1$ is larger than the random variable $U_2$.
	
	If $a>0$, then for each $\lambda>0$, using (\ref{eq:EK1}),
	\begin{align*}
		&e^{-a\,d_r(\lambda)}=E_\lambda\big[L_3(U_1)\big]\ \ge\ E_\lambda\big[L_3(U_2)\big]=L_3(0)\\
		&\quad\Longrightarrow\quad -a\,d_r(\lambda)\ \ge\ \ln L_3(0)
		\quad\Longrightarrow\quad d_r(\lambda)\ \le\ -\frac1a\ln L_3(0),
	\end{align*}
	where in the last step we divided by $-a<0$, which reverses the inequality. If $a<0$, then for each $\lambda>0$,
	\begin{align*}
		&e^{-a\,d_r(\lambda)}=E_\lambda\big[L_3(U_1)\big]\ \le\ E_\lambda\big[L_3(U_2)\big]=L_3(0)\\
		&\quad\Longrightarrow\quad -a\,d_r(\lambda)\ \le\ \ln L_3(0)
		\quad\Longrightarrow\quad d_r(\lambda)\ \le\ -\frac1a\ln L_3(0),
	\end{align*}	
	where in the last step we divided by $-a>0$, which keeps the inequality. Thus, in both cases, $d_r(\lambda)\le-\frac1a\ln L_3(0)$ for every $\lambda>0$. Using (\ref{eq:L3}) with $k=0$, we have
	$$	-\frac1a\ln L_3(0)=-p\ln 2-\frac1a\sum_{i=1}^{p}\ln\frac{\Gamma\!\left(\frac{n-i+1}{2}+a\right)}{\Gamma\!\left(\frac{n-i+1}{2}\right)}-\frac1a\ln L_4(0)=d(r),$$ where $d(r)$ is given by (\ref{eq6.3}). Assertion (ii) now follows by noting that, for each $r\in(0,1)$, $\lim\limits_{\lambda \rightarrow 0}d_r(\lambda)=d(r)$.
	
	(iii) Direct differentiation of (\ref{eq6.3}) gives
	$$\frac{\mathrm d}{\mathrm dr}d(r)=-\frac{r^{\frac{n-p}{2}-1}(1-r)^{\frac p2-1}}{a}\cdot
	\frac{\displaystyle\int_r^1 t^{\frac{n-p}{2}+a-1}(1-t)^{\frac p2-1}dt-r^a\int_r^1 t^{\frac{n-p}{2}-1}(1-t)^{\frac p2-1}dt}
	{\displaystyle\int_r^1 t^{\frac{n-p}{2}+a-1}(1-t)^{\frac p2-1}dt\cdot\int_r^1 t^{\frac{n-p}{2}-1}(1-t)^{\frac p2-1}dt}.$$
	By (\ref{eq:44L}) with $k=1$, the numerator of the last fraction has the same sign as $a$, so that the numerator divided by $a$ is positive. Hence $\frac{\mathrm d}{\mathrm dr}d(r)<0$ for both $a>0$ and $a<0$, i.e., $d(r)$ is a decreasing function of $r\in(0,1)$. Therefore, for each $r\in(0,1)$,
	$$ \lim_{r\rightarrow 1}d(r)<d(r)<\lim_{r\rightarrow 0}d(r).$$
	Using l'H\^{o}spital's rule, it follows that
	$$ \lim_{r\rightarrow 1}\frac{\int_r^1 t^{\frac{n-p}{2}+a-1}(1-t)^{\frac p2-1}dt}{\int_r^1 t^{\frac{n-p}{2}-1}(1-t)^{\frac p2-1}dt}=\lim_{r\rightarrow 1}r^{a}=1,$$
	and hence $\lim\limits_{r\rightarrow 1}d(r)=-p\ln 2-\frac1a\sum_{i=1}^{p}\ln\frac{\Gamma\!\left(\frac{n-i+1}{2}+a\right)}{\Gamma\!\left(\frac{n-i+1}{2}\right)}=c_*$. Also, note that $$ \int_0^1 t^{\frac{n-p}{2}+a-1}(1-t)^{\frac p2-1}dt=\beta\!\left(\frac{n-p}{2}+a,\frac p2\right)
	=\frac{\Gamma\!\left(\frac{n-p}{2}+a\right)\Gamma\!\left(\frac p2\right)}{\Gamma\!\left(\frac n2+a\right)} .$$ Therefore,
	\begin{align*}
		\lim\limits_{r\rightarrow 0}d(r)
		&=-p\ln 2-\frac1a\sum_{i=1}^{p}\ln\frac{\Gamma\!\left(\frac{n-i+1}{2}+a\right)}{\Gamma\!\left(\frac{n-i+1}{2}\right)}
		-\frac1a\ln\frac{\Gamma\!\left(\frac{n-p}{2}+a\right)}{\Gamma\!\left(\frac{n-p}{2}\right)}
		+\frac1a\ln\frac{\Gamma\!\left(\frac{n}{2}+a\right)}{\Gamma\!\left(\frac{n}{2}\right)}\\
		&=-p\ln 2-\frac1a\sum_{i=1}^{p}\ln\frac{\Gamma\!\left(\frac{n-i}{2}+a\right)}{\Gamma\!\left(\frac{n-i}{2}\right)}
		= c_{02},
	\end{align*}
	where the term $\frac1a\ln\frac{\Gamma(\frac n2+a)}{\Gamma(\frac n2)}$ cancels the $i=1$ term of the sum, and the remaining terms are re-indexed by $j=i-1$. Hence $c_*<d(r)<c_0$, and the result follows.
	
	For each $\lambda>0$ and $r\in(0,1)$, $R_\lambda(\delta^{\star}_{d,r})$ is minimized at $d=d_r(\lambda)$, and for each fixed $r\in(0,1)$, $\sup\limits_{\lambda>0}d_r(\lambda)=d(r)$. Using the convexity of $R_\lambda(\delta^{\star}_{d,r})$ as a function of $d$ (for fixed $\lambda>0$ and $r\in(0,1)$), we conclude that, for each $\lambda>0$ and $r\in(0,1)$, $R_\lambda(\delta_{d,r})$ is a strictly increasing function of $d$ for $d\in(d(r),\infty)$. Since $d(r)<c_0$ for each $r\in(0,1)$, the following result follows.
	
	\begin{theorem}\label{thm:41}
		Under the LINEX loss $L_2(t)$, for every $r\in(0,1)$, the estimator $\delta_{d(r),r}(\cdot,\cdot)$ in $(\ref{bz1})$, with $d(r)$ as in $(\ref{eq6.3})$, dominates the BAEE $\delta_{02}$.
	\end{theorem}
	
	Next, choose $r_0$ with $0<r<r_0<1$. As $d(\cdot)$ is decreasing on $(0,1)$, we obtain $d(r_0)<d(r)<c_{02}$. Consider now the class of estimators of the form
	\begin{align*}
		\delta^\star_{d,r_0,r}(X,S)=
		\begin{cases}
			\ln|S|+c_0, & \text{for } 0\le T\le r,\\
			\ln|S|+d(r), & \text{for } r<T\le r_0,\\
			\ln|S|+d, & \text{for } r_0<T\le 1.
		\end{cases}
	\end{align*}
	Applying the same reasoning as above, one can show that the estimator $\delta_{d(r),r}(\cdot,\cdot)$, and hence the BAEE $\delta_{02}$ is further dominated by the estimator
	\begin{align*}
		\delta^\star_{d(r_0),r_0,r}(X,S)=
		\begin{cases}
			\ln|S|+c_0, & \text{for } 0\le T\le r,\\
			\ln|S|+d(r), & \text{for } r<T\le r_0,\\
			\ln|S|+d(r_0), & \text{for } r_0<T\le 1.
		\end{cases}
	\end{align*}
	Now, using the idea of \cite{brewster1974improving}, for each $i=1,2,\ldots$ we choose a finite partition $0=r_{i,0}<r_{i,1}<\cdots<r_{i,n_i-1}<r_{i,n_i}=1$ of the interval $[0,1]$, and consider the corresponding estimator given by
	\begin{align*}
		\delta^\star_i(X,S)=
		\begin{cases}
			\ln|S|+c_0, & \text{for } r_{i,0}\le T\le r_{i,1},\\
			\ln|S|+d(r_{i,1}), & \text{for } r_{i,1}< T\le r_{i,2},\\
			\qquad\vdots & \qquad\vdots\\
			\ln|S|+d(r_{i,n_i-1}), & \text{for } r_{i,n_i-1}< T\le r_{i,n_i}.
		\end{cases}
	\end{align*}
	
	Then, provided $\max\limits_{1\le j\le n_i}|r_{i,j}-r_{i,j-1}|\to 0$ as $i\to\infty$, the sequence of estimators $\delta^\star_i(X,S)$ converges pointwise to
	\begin{equation}\label{eq:BZ}
		\delta_{\phi^\star_{BZ}}(X,S)=\ln|S|+d(T)=\ln|S+XX^t|+\phi^\star_{BZ}(T),
	\end{equation}
	where $d(x)$, $x\in[0,1]$, is defined in (\ref{eq6.3}), and $\phi^\star_{BZ}(t)=d(t)+\ln t$, $t\in[0,1]$. The second equality in (\ref{eq:BZ}) follows from the identity $\ln|S|=\ln|S+XX^t|+\ln T$.	Explicitly,
	\begin{align}
		\delta_{\phi^\star_{BZ}}(X,S)&=\ln|S|-p\ln 2-\frac1a\sum_{i=1}^{p}\ln\frac{\Gamma\!\left(\frac{n-i+1}{2}+a\right)}{\Gamma\!\left(\frac{n-i+1}{2}\right)}
		-\frac1a\ln\left[\frac{\displaystyle\int_T^1 t^{\frac{n-p}{2}+a-1}(1-t)^{\frac p2-1}dt}{\displaystyle\int_T^1 t^{\frac{n-p}{2}-1}(1-t)^{\frac p2-1}dt}\right].\label{eq:BZexplicit}
	\end{align}
	This leads to the following theorem.
	\begin{theorem}\label{thm:42}
		Suppose $d(x)$, $x\in[0,1]$, is given by $(\ref{eq6.3})$. Under the linex loss function $L_2(t)$, the estimator $\delta_{\phi^\star_{BZ}}$ defined in (\ref{eq:BZ}) dominates the BAEE $\delta_{02}$.
	\end{theorem}
	
	\noindent \textbf{Proof.} For every $i\in\mathbb{N}$, the risk of $\delta^\star_i$ is no larger than that of $\delta_{c_0}$.  As the linex loss $L_2(t)$ is non-negative, the result follows by an application of Fatou's lemma.

	We now prove that $\delta_{\phi^\star_{BZ}}$, the Brewster-Zidek-type estimator, is also generalized Bayes under a prior similar to the one used by \cite{xiaoqian2000improved} in the problem of variance estimation for the multivariate normal distribution.
	
	\begin{theorem}\label{thm:43}
		Under the linex loss function $L_2(t)$, the Brewster--Zidek-type estimator $\delta_{\phi^\star_{BZ}}$ is generalized Bayes with respect to the prior on $(\mu,\Sigma)$ whose density is
		$$\Pi(\mu,\Sigma)=\frac{1}{|\Sigma|^{\frac{p+2}{2}}}\int_0^\infty\frac{z^{\frac p2-1}}{1+z}\,e^{-\frac{n\mu^t\Sigma^{-1}\mu }{2}z}\,\mathrm dz,\qquad \mu\in\mathbb R^p,\ \Sigma>0 .$$
	\end{theorem}
	
	\noindent \textbf{Proof.}
	Under the loss $L_2(t)$, the posterior expected loss of $\delta$ is $e^{a\delta}E\big(|\Sigma|^{-a}\,\big|\,X,S\big)-a\delta+aE\big(\ln|\Sigma|\,\big|\,X,S\big)-1$. This is strictly convex in $\delta$ and is minimized at 
	\begin{align*}
		\delta_{GB3}(X,S)=-\frac1a\ln E\big(|\Sigma|^{-a}\,\big|\,X,S\big).
	\end{align*}
	The posterior density of $\Sigma$ given $(X,S)$ is
	$$ q(\Sigma\mid X,S)=\frac{\displaystyle\int_0^\infty \frac{z^{\frac p2-1}}{(1+z)^{\frac{p+2}{2}}}\left|S+\tfrac{z}{1+z}XX^t\right|^{-\frac n2} f_{n+p+1,p}\big(\Sigma\,\big|\,S+\tfrac{z}{1+z}XX^t\big)\mathrm dz} {\displaystyle\int_0^\infty \frac{z^{\frac p2-1}}{(1+z)^{\frac{p+2}{2}}}\left|S+\frac{z}{1+z}XX^t\right|^{-\frac n2}\mathrm dz},$$ where $f_{n+p+1,p}(\cdot\mid\cdot)$ is given by $	f_{\alpha, \beta}(\boldsymbol{B}|\boldsymbol{U}) =\frac{|\boldsymbol{U}|^{\frac{\alpha -\beta -1}{2}}}{2^{\frac{\beta(\alpha -\beta - 1)}{2}} \Gamma_{\beta}\left(\frac{\alpha - \beta -1}{2}\right)|\boldsymbol{B}|^{\frac{\alpha}{2}}}e^{-\frac{1}{2}\mbox{tr}(\boldsymbol{B}^{-1}\boldsymbol{U})},\ \boldsymbol{B}>0.$ For given $(X,S)$, let $ B_z\sim W_p^{-1}\!\left(n+p+1,\,S+\frac{z}{1+z}XX^t\right).$ Then $$ B_z^{-1}\sim W_p\!\left(n,\Big(S+\frac{z}{1+z}XX^t\Big)^{-1}\right)
	\ \Longrightarrow\ |B_z^{-1}|\left|S+\frac{z}{1+z}XX^t\right|\sim\prod_{i=1}^{p}\chi^2_{n-i+1},$$ where the $\chi^2_{n-i+1}$, $i=1,\ldots,p$, are independent chi-square random variables. Thus,
	\begin{align*}
		E\big(|B_z|^{-a}\big)&=\left|S+\frac{z}{1+z}XX^t\right|^{-a}\prod_{i=1}^{p}E\big((\chi^2_{n-i+1})^{a}\big)\\
		&=|S|^{-a}\left(1+\frac{z}{1+z}X^tS^{-1}X\right)^{-a}\,2^{pa}\prod_{i=1}^{p}\frac{\Gamma\!\left(\frac{n-i+1}{2}+a\right)}{\Gamma\!\left(\frac{n-i+1}{2}\right)}\\
		&=|S|^{-a}\left(1+\frac{z}{1+z}X^tS^{-1}X\right)^{-a}e^{-a\,c_*},
	\end{align*}
	and therefore
	\begin{align*}
		\delta_{G3}(X,S)&=\ln|S|+c_* -\frac1a\ln\left[\frac{\displaystyle\int_0^\infty \frac{z^{\frac p2-1}}{(1+z)^{\frac{p+2}{2}}}\left(1+\frac{z}{1+z}X^tS^{-1}X\right)^{-\left(\frac n2+a\right)}\mathrm dz}
		{\displaystyle\int_0^\infty \frac{z^{\frac p2-1}}{(1+z)^{\frac{p+2}{2}}}\left(1+\frac{z}{1+z}X^tS^{-1}X\right)^{-\frac n2}\mathrm dz}\right].
	\end{align*}
	In the above integrals, make the transformation $\left(1+\frac{z}{1+z}X^tS^{-1}X\right)^{-1}=t$ and use the fact that $(1+X^tS^{-1}X)^{-1}=T$. Then, for $b\in\{0,a\}$,
	$$ \int_0^\infty \frac{z^{\frac p2-1}}{(1+z)^{\frac{p+2}{2}}}\left(1+\frac{z}{1+z}X^tS^{-1}X\right)^{-\left(\frac n2+b\right)}\mathrm dz
	=\big(X^tS^{-1}X\big)^{-\frac p2}\int_T^1 t^{\frac{n-p}{2}+b-1}(1-t)^{\frac p2-1}dt .$$
	Substituting this, the above expression becomes $\ln|S|+d(T)$, which is the same as $\delta_{\phi^\star_{BZ}}(X,S)$ given by (\ref{eq:BZexplicit}).
	
	\section{Conclusion}\label{se6}
	In this paper, we have studied the estimation of the differential entropy of a multivariate normal distribution with unknown mean vector and unknown covariance matrix. We have considered this problem under a general location-invariant loss function that is strictly convex and minimized at zero. We have characterized the class of affine equivariant estimators and derived the BAEE. Explicit expressions of the BAEE are obtained for the squared error and linex loss functions, along with its bias and risk. Further, we have shown that the BAEE is generalized Bayes with respect to the improper prior $\pi(\mu,\boldsymbol{\Sigma})=|\boldsymbol{\Sigma}|^{-(p+1)/2}$ under both the loss functions. We have proved that the BAEE is inadmissible under the general loss function by deriving Stein-type improved estimators that exploit the information contained in the sample mean. These results are also extended to the case of a matrix variate normal distribution with unknown mean matrix. Under the linex loss function, we have obtained a smooth Brewster--Zidek-type estimator that dominates the BAEE, and we have shown that this estimator is generalized Bayes with respect to a suitable non-informative prior.

	\section*{Data availability statement}
	Data sharing is not applicable to this article as no datasets were analyzed during the current study.
	\section*{Disclosure statement}
	No potential conflict of interest was reported by the authors.
	
\end{document}